\documentclass[12pt]{article}

\usepackage{amsfonts}
\def\be{\begin{equation}}
\def\ee{\end{equation}}
\def\l{\langle}
\def\r{\rangle}

\def\R{I\!\!R}
\def\N{I\!\!N}

\def\bea{\begin{eqnarray}}
\def\eea{\end{eqnarray}}

\def\cR{{\cal R}}
\def\cN{{\cal N}}

\def\12{\frac{1}{2}}

\newtheorem{remark}{Remark}

\newtheorem{definition}{Definition}
\newtheorem{proposition}{Proposition}

\newenvironment{proof}{\noindent\textbf{Proof.}
  }{\hspace*{\fill}$\spadesuit$ \\[2mm]}

\begin{document}
\begin{center}
{\Large \bf A note on Kaczmarz algorithm  with remotest set control sequence}

\vspace*{1cm}
{\small CONSTANTIN POPA}

\vspace*{1cm}
Ovidius University of Constanta, Blvd. Mamaia 124, Constanta 900527, Romania; {\tt cpopa@univ-ovidius.ro}

\vspace*{0.2cm}
              ``Gheorghe Mihoc - Caius Iacob'' Institute of Statistical Mathematics and Applied Mathematics of the Romanian Academy, Calea 13 Septembrie, Nr. 13, Bucharest 050711, Romania\\
\end{center}

{\bf Abstract.} 
In this paper we analyse the Kaczmarz projection algorithm with Remotest set and Random control of projection indices and provide a sufficient condition such that  each projection index appears infinitely many times during the iterations.

{\bf Keywords:} {Kaczmarz algorithm;  remotest set control}

{\bf MSC (2010):} {65F10;  65F20}

\section{Introduction}
\label{intro}

For   an $m \times n$ (real) matrix $A$ and  $b \in \R^m$  let
\begin{equation}
\label{1}
 Ax = b
\end{equation}
be a consistent system of linear equations and denote by  $S(A; b), x_{LS}$ the set of its solutions and  the minimal (Euclidean)  norm one ($\left\langle \cdot, \cdot \right\rangle$ and $\parallel \cdot \parallel$ will denote the Euclidean scalar product and norm on some space $\R^q$, respectively). Other notations used will be $A^T, A_i, A^j, \cR(A), \cN(A)$, ${rank}(A)$ for the transpose, $i$-th row, $j$-th column, range, null space and rank of $A$. The projection onto a nonempty closed convex set  $V$ will be denoted by $P_V$, and for $V=H_i = \left\{x \in \R^n, \left\langle x, A_i \right\rangle = b_i \right\}$ (the hyperplane determined by the $i$-th equation of the  system (\ref{1})) we know that
\begin{equation}\label{res:PHi}
P_{H_i}(x) = x - \frac{\left\langle x, A_i \right\rangle - b_i}{\left\Vert A_i \right\Vert ^2}A_i.
\end{equation}
 The Kaczmarz's  iterative method for numerical solution of  (\ref{1}) has the form from below.\\
{\bf Algorithm K.} \\
{\it Initialization:} $x^0 \in \R^n$ \\
{\it Iterative step:} for $k = 0, 1, \dots$ select $i_k \in \{ 1, 2, \dots, m \}$ and compute $x^{k+1}$ as 
        \be
        \label{gen}
        x^{k+1}= P_{H_{i_k}}(x^k).
        \ee
 There have been defined several classes of selection procedures for the indices $i_k$ (see \cite{censta, ycrow, ccp11, comb} and references therein). In this paper we will consider the Maximal Residual (remotest set) control and the Random control procedures, and provide a sufficient condition such that in the case of Kaczmarz's projection method  {\bf K}, they belong to the class of {\it control} selections from \cite{ccp11}. The paper is organized as follows: in section \ref{control} we give an equivalent formulation of the {\it  control} sequence definition from \cite{ccp11}. In section \ref{maxrez} we show that, for $x^0=0$ and  under  additional assumptions, the Maximal Residual (remotest set) selection or the Random selection is a {\it  control} w.r.t. section \ref{control}.

\section{Control sequences}
\label{control}

Let $\N$  denote the set of natural numbers $\{ 0, 1, 2, \dots, \}$. In \cite{ccp11} the following definition concerning control sequences was introduced.
\begin{definition}{\bf (D1)}
\label{def1}
Given a monotonically increasing sequence $\{\tau_{k}\}_{k \geq 0} \subset \N$, a mapping $i: \N \rightarrow \{1, 2, \dots, m\}$ is called a {\em control with respect to the sequence $\{\tau_{k}\}_{k \geq 0}$} if it defines a {\em  sequence $\{i(t)\}_{t \geq 0}$}, such that for all $k \geq 0$,
\begin{equation}\label{d1}
\{1, 2, \dots, m\} \subseteq \{i(\tau_k), i(\tau_k + 1), \dots, i(\tau_{k + 1} - 1)\}.
\end{equation}
\end{definition}
The next definiton, mentioned in \cite{bau} (see also \cite{222}) points-out on an important aspect of control sequences.
\begin{definition}{\bf (D2)}
\label{def2}
 A  mapping $i: \N \rightarrow \{1, 2, \dots, m\}$ is called a {\em random mapping} if any $i \in \{ 1, \dots, m \}$ appears infinitely many times in the set ${\cal{I}} = \{  i(k), k \geq 0 \}$.
\end{definition}
It is clear that, if the mapping $i$ is a control with respect to some sequence $\{\tau_{k}\}_{k \geq 0}$, then it is also a random mapping, according to  definition {(\bf D2)}. Indeed, if the sequence $\{\tau_{k}\}_{k \geq 0}$ is increassing then $\tau_{k+1} > \tau_k$ and the sets 
$\Delta_k = \{i(\tau_k), i(\tau_k + 1), \dots, i(\tau_{k + 1} - 1)\}, k \geq 0$ form a partition of $\N$ as in (\ref{d1}). 
Next proposition tell us about the reciprocal of this property, i.e. a random mapping is a  control according to the definition {(\bf D1)}.
\begin{proposition}
\label{d2d1}
Let ${\rm i}: \N \rightarrow \{1, 2, \dots, m\}$ be  a random mapping (according to {\bf (D2)}). Then it exists a  monotonically increasing sequence $\{\tau_{k}\}_{k \geq 0} \subset \N$  such that ${\rm i}$ is a    control  w.r.t. {\bf (D1)}.
\end{proposition}
\begin{proof}
We will first write {\bf (D2)} in the following equivalent formulation: for any $i \in \{ 1, \dots, m \}$ it is true that
\begin{equation}\label{d2}
\forall ~ k \geq 1, ~\exists k_i \geq k ~{\rm s.t.}~  i_k = i.
\ee
We will now recursively define an increasing sequence $\{\tau_{k}\}_{k \geq 0} \subset \N$ as follows: for $k=0$ we set $\tau_0 = 0$; for $k=1$ let $\tau_1$ be the smallest natural number with the properties
\begin{equation}\label{d3}
\tau_1 >  \tau_0 ~{\rm and}~ \{1, 2, \dots, m\} \subseteq \{ i(\tau_0), \dots, i(\tau_1 - 1) \}.
\ee
Such a number $\tau_1$ exists according to the equivalent formulation  (\ref{d2}). In general, if we already have constructed $\tau_k$, then $\tau_{k+1}$ will be the smallest natural number such that 
\begin{equation}\label{d4}
\tau_{k+1} >  \tau_k ~{\rm and}~ \{1, 2, \dots, m\} \subseteq \{ i(\tau_k), \dots, i(\tau_{k+1} - 1) \},
\ee
 which is exactly the property  (\ref{d1}) of definition {\bf (D1)}, and the proof is complete.
\end{proof}
Based on the above proposition we will consider in the rest of the paper  as definition for  controls the equivalent formulation from   {\bf (D2)}. In this respect, the following two selection procedures will be analysed.
\begin{itemize}
\item {\bf Maximal Residual (remotest set) (\cite{a84}):} Select $i_{k} \in \{ 1, 2, \dots, m \}$ such that
\be\label{MR}
|\l A_{i_k}, x^{k-1} \r - b_{i_k}| = \max_{1 \leq i \leq m} |\l A_{i}, x^{k-1} \r - b_{i}|.
\ee
\item {\bf Random (\cite{strohm}):} Let the set $\Delta_m \subset \R^m$ be defined by  
\be
\label{0010}
\Delta_m = \{ x \in \R^m, x \geq 0, \sum_{i=1}^m = 1 \},
\ee
define the discrete probability distribution 
\begin{equation}\label{RK}
p \in \Delta_{m},\ p_{i} = \frac{\|A_{i}\|^{2}}{\|A\|^{2}_{F}}, 
\ i = 1, \dots, m,
\end{equation}
and select $i_k \in \{ 1, 2, \dots, m \}$ such that
\be\label{RK1}
i_{k} \sim p .
\ee 
\end{itemize}
The main aspect regarding the above two selection procedures  is concerned with the fact that the projection indices are generated recursively,  without no a priori information  on them. And, at least related to author's knowledge, there are no results saying that when the algorithm {\bf K} is applied with one or the other of  the above selection procedures, each projection index will appear infinitely many times.  
\section{The  Kaczmarz algorithm}
\label{maxrez}

We consider in this section Kaczmarz's projection algorithm  in which the Maximal Residual (remotest set) (\ref{MR}) or Random (\ref{RK})-(\ref{RK1})  procedure is used for selecting the projection indices in each iteration, and  with the initial approximation  $x^0=0$. 
In this case, in papers \cite{a84} and \cite{strohm} it is proved that the sequence $(x^k)_{k \geq 0}$ generated by algorithm {\bf K} converges to the minimal norm solution $x_{LS}$ of the system (\ref{1}). 
We will formulate a sufficient condition such that any of the above  selection procedures  satisfies {\bf (D2)}). For  $i \in \{ 1, \dots, m \}$  arbitrary fixed, let 
  ${A}^{(i)}: (m-1) \times n$, ${b}^{(i)} \in \R^{m-1}$ be the submatrix of $A$ without the $i$-th row, respectively the subvector of $b$ without the $i$-th component and  ${x}^{(i)}_{LS}$  the minimal norm solution of the  system  ${A}^{(i)} {x} = {b}^{(i)}$. \\
 {\bf Assumption C.} For any index $i \in \{ 1, \dots, m \}$ we have 
\be
\label{n11}
x_{LS} ~\neq~  {x}^{(i)}_{LS}.
\ee
\begin{proposition}
\label{pp1}
If the assumption {\bf C} holds, then any of the above two selection procedures within the Kaczmarz's iteration {\bf K} satisfies {\bf (D2)}.
\end{proposition}
\begin{proof}
Let us suppose that the conclusion of the proposition is not true. According to  (\ref{d2}) it exists an index $i_0 \in \{ 1, \dots, m \}$ and an integer $k_0 \geq 1$ such that, in the selection procedure of the {\bf K} algorithm iterations we have
\be
\label{n1}
i_k \neq i_0, ~\forall k \geq k_0.
\ee
Therefore,  the sequence $(x^k)_{k \geq k_0}$ is generated by the {\bf K} algorithm applied (only !) to the subsystem ${A}^{(i_0)} {x} = {b}^{(i_0)}$. By the theory from \cite{a84} and \cite{strohm}, respectively, it results that
\be
\label{n2}
\lim_{0 \geq k \rightarrow \infty} x^k =  x_{LS} = ~\lim_{k_0 \geq k \rightarrow \infty} x^k =  {x}^{(i)}_{LS},
\ee
hence
\be
\label{n3}
 x_{LS} =  {x}^{(i_0)}_{LS},
\ee
which contradicts (\ref{n11}) and completes the proof.
\end{proof}
In order to understand what means a condition like (\ref{n11}) we will analyse it in the particular case 
\be
\label{n3p}
m \leq n ~{\rm and}~ rank(A)=m,
\ee
 for which the system (\ref{1}) is consistent for any $b \in \R^m$. In this case for any index $i$  the matrix ${A}^{(i)}$ is also full-row rank, and  the system ${A}^{(i_0)} {x} = {b}^{(i_0)}$  also consistent. We will arbitrary fix the index $i \in \{ 1, \dots, m \}$ and denote by  $\tilde{A}, \tilde{b}, \tilde{x}_{LS}$ the elements  
${A}^{(i)}, {b}^{(i)}, {x}^{(i)}_{LS}$, respectively. Moreover, we will analyse the opposite assumption of  (\ref{n11}), namely 
\be
\label{n11p}
x_{LS} ~=~  \tilde{x}_{LS}.
\ee
What does this mean in terms of the matrix $A$ and right hand side $b$ ? For simplfying the presentation we will suppose that $i=m$ (this assumption is not too restrictive because it can be obtained by a row-permutation in $A$ and $b$, which does not affect the spectral properties of $A$ and the  solution set $S(A; b)$). Because $A^T$ is overdetermined and full-column rank, there exist an $n \times n$ orthogonal matrix $Q$ and the QR decomposition
$$
Q^T A^T = \left [ \begin{array}{c} R \\ 0 \\ \end{array} \right ] = 
\left [
\begin{array}{ccccc}
r_{11} & r_{12} & \dots & r_{1, m-1} & r_{1m}  \\
0  & r_{22} & \dots & r_{2, m-1} & r_{2m}  \\
\dots  & \dots & \dots & \dots & \dots  \\
 0 & 0 & \dots & r_{m-1, m-1} & r_{m-1, m} \\
  0 & 0 & \dots & 0 & r_{mm} \\
    0 & 0 & \dots & 0 & 0 \\
        \vdots &  &  & \vdots & \vdots \\
      0 & 0 & \dots & 0 & 0 \\
\end{array}
\right] = 
$$
\be
\label{n5}
 Q^T [ \tilde{A}^T | A_m ] = [ Q^T \tilde{A}^T | Q^T A_m ] = 
\left [ \begin{array}{ccc} 
\tilde{R} & & c \\ 
0 & & r_{mm}\\ 
0 & \dots & 0 \\
\vdots & & \vdots \\
0 & \dots & 0\\
\end{array} \right ] .
\ee
Therefore
\be
\label{n6}
Q^T \tilde{A}^T = \left [ \begin{array}{c} \tilde{R} \\ 0 \\ \end{array} \right ]
\ee
will be a QR decomposition for $\tilde{A}$, where
\be
\label{n7}
\tilde{R} = \left [
\begin{array}{cccc}
r_{11} & r_{12} & \dots & r_{1, m-1}   \\
0  & r_{22} & \dots & r_{2, m-1}   \\
\dots  & \dots & \dots & \dots  \\
 0 & 0 & \dots & r_{m-1, m-1}  \\
\end{array}
\right] ~{\rm and}~ c=(r_{1m}, r_{2m}, \dots, r_{m-1, m})^T.
\ee
Because $m \leq n$ and $A, \tilde{A}$ have full-row rank, we know that (see e.g. \cite{cpbook})
$$
A^+ = A^T (A A^T)^{-1}, ~~\tilde{A}^+ = \tilde{A}^T (\tilde{A} \tilde{A}^T)^{-1},
$$ hence 
\be
\label{n8}
x_{LS} = A^+ b = Q \left [ \begin{array}{c} {R^{-T}} b\\ 0 \\ \end{array} \right ], ~~\tilde{x}_{LS} = \tilde{A}^+ \tilde{b} = Q \left [ \begin{array}{c} \tilde{R}^{-T} \tilde{b} \\ 0 \\ \end{array} \right ] .
\ee
In our hypothesis (\ref{n11p}), and by using (\ref{n5}) and (\ref{n7}) we get from (\ref{n8}) the equality
\be
\label{n9}
{R^{-T}} b = \left [ \begin{array}{c} \tilde{R}^{-T} \tilde{b}  \\
 0  \\ \end{array} \right ], 
~{\rm where}~ R^T = \left [ \begin{array}{cc} 
\tilde{R}^{T}  & 0 \\
 c^T & r_{mm}\\ \end{array} \right ]: m \times m.
\ee
It can be easily shown that 
\be
\label{n10}
R^{-T} = (R^T)^{-1} =  \left [ \begin{array}{cc} 
\tilde{R}^{-T}  & 0 \\
 -\frac{1}{r_{mm}} c^T \tilde{R}^{-T} & \frac{1}{r_{mm}}\\ \end{array} \right ],
\ee
which together with the first equality in (\ref{n9}) gives 
$$
\left [ \begin{array}{c} \tilde{R}^{-T} \tilde{b}  \\
 0  \\ \end{array} \right ] = \left [ \begin{array}{cc} 
\tilde{R}^{-T}  & 0 \\
 -\frac{1}{r_{mm}} c^T \tilde{R}^{-T} & \frac{1}{r_{mm}}\\ \end{array} \right ] \left [ \begin{array}{c} \tilde{b}\\ b_m \\ \end{array} \right ] = 
$$
$$
\left [ \begin{array}{c} 
\tilde{R}^{-T} \tilde{b}  \\
 -\frac{1}{r_{mm}} c^T \tilde{R}^{-T} \tilde{b} + \frac{b_m}{r_{mm}}\\ \end{array} \right ],
$$
and therefore
$$
0 ~=~ -\frac{1}{r_{mm}} c^T \tilde{R}^{-T} \tilde{b} + \frac{b_m}{r_{mm}}
 ~~{\rm or}~~ c^T \tilde{R}^{-T} \tilde{b} = b_m.
$$
Eventually, we proved that if (\ref{n11p}) holds (for $i=m$), then
\be
\label{n12}
 c^T \tilde{R}^{-T} \tilde{b} = b_m,
\ee
where the elements $c, \tilde{R}$ are from (\ref{n5}) and $\tilde{b}$ is the right hand side of the system ${A}^{(i_0)} {x} = {b}^{(i_0)}$. But, also the converse holds, namely: if (\ref{n12}) is true with the above elements, then
(\ref{n11p}) holds (for $i=m$). This is true if we assume that  in the QR decomposition (\ref{n5}) the diagonal elements satisfy $r_{ii} > 0, \forall i$, which gives us the unicity of the factor $R$ in the QR decomposition.
\begin{remark}
\label{frem}
Although the assumption $x^0=0$ in {\bf K} is essential for the proof of Proposition \ref{pp1}, we conjecture  that this result is stil true for a larger class of initial approximations $x^0$.  Unfortunately we do not  have for the moment a theoretical proof in this respect. 
\end{remark}
{\bf Acknowledgements.} We would like to thanks to Prof. Yair Censor for his very helpful comments that have much improved the initial version of the paper.

\end{document}